\documentclass{amsart}

\newtheorem{theorem}{Theorem}[section]
\newtheorem{lemma}[theorem]{Lemma}

\theoremstyle{definition}
\newtheorem{definition}[theorem]{Definition}
\newtheorem{pro}[theorem]{Proposition}

\usepackage{tikz}
\theoremstyle{remark}
\newtheorem{remark}[theorem]{Remark}
\numberwithin{equation}{section}

\begin{document}
\title[ Energy bound for
 Kapustin-Witten solutions on $S^3\times\mathbb{R}^+$]{ Energy bound for\\
 Kapustin-Witten solutions on $S^3\times\mathbb{R}^+$}
\author{Naichung Conan Leung$\mbox{ }^*$, Ryosuke Takahashi$\mbox{ }^\dagger$}
%    Address of record for the research reported here
\address{$\mbox{ }^*$ Institute of Mathematical Scences, Chinese University of Hong Kong,
Academic Building No.1, CUHK, Sha Tin, New Territories, HK}
\email{ leung@math.cuhk.edu.hk }
\address{$\mbox{ }^\dagger$ Department of mathematics, National Cheng Kung University, No.1, University Road, East District, Tainan, Taiwan }
\email{ tryotriple@gmail.com}

\maketitle
\begin{abstract} We show that for solutions of Kapustin-Witten equation with Nahm pole boundary condition on $S^3\times \mathbb{R}^+$, there exists a universal constant bound on its Yang-Mills energy $\|F_A\|_{L^2}$. 
\end{abstract}
\section{Introduction}

Let $M=S^3\times \mathbb{R}^+$ be the Riemannian manifold with the standard metric, $P$ be a $SU(2)$-principal bundle. For any $(A,\phi)\in Conn(P)\times \Omega^1(M;\mathfrak{su}(2))$, the Kapustin-Witten equation [6] is written as follows
\begin{align}
F_A-\phi \wedge \phi&=*d_A \phi;\\
d_A *\phi&=0.\nonumber
\end{align}
For this elliptic PDE, because $M$ is noncompact, a suitable boundary condition is needed to study its solution space. In [6], [7], Mazzeo and Witten considered the asymptotic boundary condition associated with a knot (or link) $K\subset S^3\times \{0\}$. \footnote{To be more precise, they studied the local models defined on $\mathbb{R}^3\times \mathbb{R}^+$, which can be also regarded as the local model for $Y\times \mathbb{R}^+$ for any 3-manifold $Y$.} This condition requires solutions asymptotic to Nahm pole solution near any boundary point $(p,0)\in (S^3-K)\times\{0\}$ and asymptotic to Nahm pole singular solution near $(p,0)\in K\times\{0\}$. These two model solutions will also be described in Section 2.1. We use the notation $(NP)_K$ to denote this boundary condition.\\

Let $y$ be the variable on $\mathbb{R}^+$. Besides the boundary condition defined on $y=0$, an asymptotic boundary condition is also required as $y$ goes to infinity. One reasonable condition given by Mazzeo and Witten [6] assumes the deviation of $(A,\phi)$ from the Nahm pole solution is of order $\frac{1}{y}$ as well as $F_A$ and $|\nabla_A\phi|$ are of the order $\frac{1}{y^2}$.\footnote{In their case defined on $\mathbb{R}^3\times \mathbb{R}^+$, the asymptotic behavior as $x$ goes to infinity is also needed. So the condition they assumed is that the deviation of $(A,\phi)$ from the Nahm pole solution is of order $\frac{1}{r}$ as well as $F_A$ and $|\nabla_A\phi|$ are of the order $\frac{1}{r^2}$, where $r=\sqrt{|x|^2+y^2}$}\\

Another reasonable condition is developed by works of Taubes and He. In [8], Taubes considered $PSL(2;\mathbb{C})$-connections $\mathbb{A}=A+i\phi$    
with $(A,\phi)$ satisfies KW equation. Under this setting, a complexified curvature can be defined as
\begin{align*}
F_{\mathbb{A}}=F_A-\phi^2+id_A\phi.
\end{align*}
So we call $(A,\phi)$ a $PSL(2;\mathbb{C})$ flat connection if $F_{\mathbb{A}}=0$. In [4], S. He proved that for any solution to (1.1) which converges to a flat $PSL(2;\mathbb{C})$-connection in $C^2$-sense will decay exponentially as $y$ goes to infinity. Therefore, this condition is an alternative one for replacing the condition provided by Mazzeo and Witten. In this paper, we use this as our asymptotic boundary condition as $y$ goes to infinity. i.e. $(A,\phi)$ converges to the $PSL(2;\mathbb{C})$ flat connection in $C^2$-sense.\footnote{ In fact, in [4], it is sufficient for $(A,\phi)$ converging to a flat connection in $L^2_2$-sense.}\\

Under this setting, for any link $K$ in $S^3$, we define the moduli space $\mathfrak{M}_K$ to be the collection of $(A,\phi)$ satisfy (1.1), $(NP)_K$ and asymptotic to the flat $PSL(2;\mathbb{C})$-connection as $y$ goes to infinity. In this paper, we prove the following theorem.
\begin{theorem}
For $K=\emptyset$, there exists $C>0$ such that
\begin{align*}
\int_{S^3\times \mathbb{R}^+}|F_A|^2\leq C 
\end{align*}
for all $(A,\phi)\in\mathfrak{M}_{K=\emptyset}$
\end{theorem}

To motivate readers for studies of $\mathfrak{M}_K$ and Theorem 1.1, we should start with the conjecture given by E. Witten. In [9], Witten conjectured:\\
\ \\
{\bf Witten Conjecture}. The generating function of the numbers of Kapustin-Witten solutions on $\mathbb{R}^3\times\mathbb{R}^+$ satisfying $(NP)_K$ is Jones polynomial.\\
\ \\
 To formulate this conjecture, we need the following settings first. For any knot or link $K\subset \mathbb{R}^3$,
\begin{align*}
\bullet &\mbox{ Denote by }\mathfrak{M}_K^{\mathbb{R}^3} \mbox{ the space of solutions of KW-equation on }\mathbb{R}^3\times \mathbb{R}^+, (NP)_K;\\&\mbox{the deviation of }(A,\phi)\mbox{ from the Nahm pole solution is of order }\frac{1}{\sqrt{|x|^2+y^2}};\\
& F_A, |\nabla_A\phi| \mbox{ are of the order } \frac{1}{|x|^2+y^2}.\\
\bullet &\mbox{ The topological charge is defined to be }p(A):=\frac{1}{4\pi^2}\int_M tr(F_A)^2. \mbox{ Any  finite}\\
&\mbox{ difference of this topological charge between any two elements in }\mathfrak{M}_K^{\mathbb{R}^3}\mbox{ is}
\\
& \mbox{  an integer}.
\end{align*}
Here the second bullet is just a straightforward result of Chern-Simons theory, see [1]. The interesting conjecture made by Witten involves the relation between $\mathfrak{M}_K^{\mathbb{R}^3}$ and Jones polynomials, which is the following. By choosing a suitable number $c$ such that $p(A)-c\in \mathbb{Z}$, we can define $l_n$ to be the counting of solutions, with suitable sign, in $\{(A,\phi)\in\mathfrak{M}_K|p(A)-c =n\}$ for any $n\in \mathbb{Z}$. Then the polynomial $J(t):=\sum l_n t^n$ will be the Jones polynomial for the knot $K$. It is also conjectured by Witten that the moduli space $\mathfrak{M}_K^{\mathbb{R}^3}$ is related to the categorification of the Jones polynomials, namely the Khovanov homology.\\

The first difficulty of proving Witten's conjecture is the compactness problem for $\mathfrak{M}_K^{\mathbb{R}^3}$. To attack this problem, the first important question is to find some energy bound for the space of solutions. It is also natural to believe this is true for $\mathfrak{M}_K$. Theorem 1.1 shows that in this case $\int|F_A|^2$ has a universal bound when $K=\emptyset$. Meanwhile, by the conjecture and the structure of Jones polynomials, $p(A)$ is finite. This is also implied by our theorem when $K=\emptyset$.\\

The first part of the proof of Theorem 1.1 uses the Weitzenb\"{o}ck formula technique given by Mazzeo and Witten [6]. The main difficulty is: When we apply the formula, a Ricci curvature term $Ric(\phi,\phi)$ involved here which has order $y^{-2}$. In particular, it is not integrable.\\

The proof of Theorem 1.1 also . One might expect that Theorem 1.1 will continue to hold for manifold $Y\times \mathbb{R}^+$ as long as $Y$ equipped a metric with positive Ricci curvature. However, the argument in this paper relies on the presence of an exact solution provided by S. He [3], which is absent in general. Meanwhile, the positivity of Ricci curvature defined on $Y$ seems to be necessary for the boundness of $\|F_A\|_{L^2}^2$. Note that there are counter examples to Theorem 1.1 false if $Y$ has negative Ricci curvature, given by He and Mazzeo in [5].

\section{Nahm pole solutions and asymptotic boundary condition}
To describe the boundary condition, we need to introduce the Nahm pole boundary condition studied in [6] and Nahm pole singular boundary condition studied in [7]. This gives us the formal definition of $(NP)_K$.\\

First of all, we will introduce two solutions $(A^\emptyset,\phi^\emptyset)$, $(A^{L},\phi^{L})$ of KW-equation defined on $\mathbb{R}^3\times \mathbb{R}^+$, see [6] and [7]. Then we call a solution $(A,\phi)$ satisfies $(NP)_K$ if and only if it converges to one of them by taking scaling limit.

\subsection{Model solutions and asymptotic boundary condition}
First solution we will introduce is the Nahm pole solution. Let $\{\mathfrak{t}_i\}_{i=1,2,3}$ be the generators of $\mathfrak{su}(2)$ which satisfy $[\mathfrak{t}_i,\mathfrak{t}_j]=\epsilon_{ijk}\mathfrak{t}_k$. Here $\epsilon_{ijk}$ is the permutation sign. Meanwhile, we use the coordinate $(x_1,x_2,x_3,y)$ to parametrize $\mathbb{R}^3\times \mathbb{R}^+$. Under these notations, we define
\begin{align*}
(A^{\emptyset},\phi^{\emptyset})=(0,\sum_i\frac{\mathfrak{t}_i}{y}dx_i).
\end{align*}
It is easy to check that this solution satisfies (1.1) on $\mathbb{R}^3\times \mathbb{R}^+$. Moreover, in [6], Mazzeo and Witten proved that there is no other solution of KW-equation asymptotic to $(A^{\emptyset},\phi^{\emptyset})$. This model solution will be used to describe the asymptotic behavior at those points on $S^3-K$.\\

The second model was initially defined by Witten in [9] (also appeared in [7]) which is usually called Nahm pole singular solution. In the paper, Witten constructed a family of solutions, $(A^{\tau},\phi^{\tau})$, satisfying KW equation with singularity along $x_3$-axis for any nonnegative integer $\tau$. In this paper, we just consider the simplest case with $\tau=1$ ($\tau=0$ is corresponding to the case $(A^{\emptyset},\phi^{\emptyset})$). Then the solution can be written in the following way. Let $x_1^2+x_2^2=r^2$, we define
\begin{align*}
&\mathfrak{f}_1=\frac{x_1}{\sqrt{r^2+y^2}}\mathfrak{t}_1-\frac{x_2}{\sqrt{r^2+y^2}}\mathfrak{t}_2,\\
&\mathfrak{f}_2=\frac{x_1}{\sqrt{r^2+y^2}}\mathfrak{t}_2+\frac{x_2}{\sqrt{r^2+y^2}}\mathfrak{t}_1,\\
&\mathfrak{f}_3=\Big(1+\frac{y^2}{r^2+y^2}\Big)\mathfrak{t}_3.
\end{align*}
The Nahm pole singular solution will be
\begin{align*}
(A^S,\phi^S)=\Big(-\frac{x_2}{r^2+y^2}\mathfrak{t}_3dx_1+\frac{x_1}{r^2+y^2}\mathfrak{t}_3dx_2,\sum_i\frac{\mathfrak{f}_i}{y}dx_i\Big)
\end{align*}
defined on $\mathbb{R}^3\times\mathbb{R}^+$. This model solution will be used to describe the asymptotic behavior at those points on $K$ with the tangent of $K$ parallel to $x_3$-axis.\\

To explain the boundary condition in [6] and [7], we use the following notation. For any $p\in S^3$, one can use the spherical projection to define the local coordinate chart $\varphi_p:\mathbf{B}_1\rightarrow U_p\subset S^3$ where $U_p$ is the open ball centered at $p$. We call this coordinate chart {\bf standard} if $\partial_{x_3}$ is parallel to the tangent of $K$ whenever $p\in K$.

\begin{definition}
Let $(A,\phi)\in Conn(P)\times \Omega^1(M;\mathfrak{su}(2))$ and $p\in S^3$. We say that $(A,\phi)$ satisfies $(NP)_K$ if and only if the following two conditions hold. Firstly there exists $g\in SU(2)$ such that
\begin{align}
\lim_{s\rightarrow 0} \varphi^*_p(A,\phi)(s\vec{x},sy)= g(A^\emptyset,\phi^\emptyset)g^{-1}(\vec{x},y)\mbox{  when }p\in S^3-K.
\end{align}
Secondly, there exists $g\in SU(2)$ such that
\begin{align}
\lim_{s\rightarrow 0}\varphi^*_p (A,\phi)(s\vec{x},sy)=g(A^S,\phi^S)g^{-1}(\vec{x},y)\mbox{  when }p\in K \mbox{ and }\varphi_p\mbox{ {\bf standard}}.
\end{align}
\end{definition}

To define the asymptotic behavior as $y\rightarrow \infty$, by the observation in [8] (also appeared in [4]), one can regard  $A+i\phi$ as a $PSL(2;\mathbb{C})-$connection. We suppose that all the solutions we considered converges to a flat $PSL(2;\mathbb{C})-$connection, i.e. $(F_A-\phi^2,d_A\phi)\rightarrow 0$ as $y\rightarrow \infty$ in $C^2$-sense. We define the moduli space $\mathfrak{M}_K$ as the following.

\begin{definition}
\begin{align*}
\mathfrak{M}_K:=\{(A,\phi)\in &Conn(P)\times \Omega^1(S^3\times\mathbb{R}^+;\mathfrak{su}(2))|\\
&(A,\phi)\mbox{ satisfies Kapustin-Witten equation and }(NP)_K,\\
&\mbox{ and converges to a flat }PSL(2,\mathbb{C})\mbox{ connection as }y\rightarrow \infty \},\\
\hat{\mathfrak{M}}_K:=\mathfrak{M}_K/\mathcal{G}.
\end{align*}
Here $\mathcal{G}=C^{\infty}(S^3\times\mathbb{R}^+;SU(2))$ is the gauge group with the transformation 
\begin{align*}
{\bf g}\cdot(A,\phi)=({\bf g}A{\bf g}^{-1}+(d{\bf g}){\bf g}^{-1}, {\bf g}\phi{\bf g}^{-1}).
\end{align*}
\end{definition}

In this paper, we only consider the case that $K=\emptyset$. The elements in $\mathfrak{M}_{K=\emptyset}$ are usually called Nahm pole solutions. Mazzeo and Witten proved the following proposition in [6].
\begin{pro}
For any two solutions $(A_0,\phi_0)$, $(A_1,\phi_1)\in \mathfrak{M}_{K=\emptyset}$, there exists ${\bf g}\in \mathcal{G}$ such that
\begin{align}
(A_0,\phi_0)-{\bf g}\cdot(A_1,\phi_1)=\Big(\sum_{i=2}^{\infty}a_i(\vec{x})y^i, \sum_{i=1}^{\infty}q_i(\vec{x})y^i\Big)
\end{align}
for some smooth $\mathfrak{su}(2)$-valued 1-forms $a_i,q_i$ defined on $S^3\times \mathbb{R}^+$.
\end{pro}

Finally, we introduce an important model solution constructed by S. He. In [3], He constructed two solutions of KW equation. Both of them satisfy the asymptotic boundary condition $(NP)_{K=\emptyset}$ and are asymptotically flat. We consider one of them, which can be written as:
\begin{align*}
(A^H,\phi^H)=\Bigg(\frac{6e^{2y}}{e^{4y}+4e^{2y}+1}\omega,\frac{6(e^{2y}+1)e^{2y}}{(e^{4y}+4e^{2y}+1)(e^{2y}-1)}\omega\Bigg).
\end{align*}
where 
\begin{align}
\omega=\sum_{i=1}^3 \mathfrak{t}_i\mathfrak{e}_i
\end{align}
with $\{\mathfrak{e}_i\}$ be the orthonormal co-frame on $S^3$.
\begin{remark}
One can also choose the other solution in [3], which can be written as 
\begin{align*}
\Big(\frac{2(e^{4y}+e^{2y}+1)}{e^{4y}+4e^{2y}+1}\omega,\frac{6(e^{2y}+1)e^{2y}}{(e^{4y}+4e^{2y}+1)(e^{2y}-1)}\omega\Big).
\end{align*}
This solution is also asymptotic to the flat connection under a gauge transformation. So we can prove Theorem 1.1 by using either one of them.
\end{remark}

One can easily check the following property for this model solution.
\begin{pro}
$C_H:=\|F_{A^H}\|_{L^2}+\|dy\wedge\partial_y\phi^H-*({\phi^H})^2\|_{L^2}<\infty$.
\end{pro}
\begin{proof}
One can compute directly to obtain
\begin{align*}
F_{A^H}=dA^H+A^H\wedge A^H=\frac{12(e^{2y}-e^{6y})}{(e^{4y}+4e^{2y}+1)^2}dy\wedge \omega+\frac{36e^{4y}}{(e^{4y}+4e^{2y}+1)^2}\omega^2.
\end{align*}
Since $(A^H,\phi^H)$ satisfies Kapustin-Witten equation, we have
\begin{align*}
dy\wedge\partial_y\phi^H-*({\phi^H})^2=*\Bigg(\frac{36e^{4y}}{(e^{4y}+4e^{2y}+1)^2}\omega^2\Bigg).
\end{align*}
So both of them are in $L^2$.
\end{proof}

\section{Proof of the main theorem}
\subsection{Gauge transformations and Taubes' maximum principle}
Let $(A,\phi)\in \mathfrak{M}_{K=\emptyset}$. By Proposition 2.3, there exists a gauge transformation ${\bf g}\in C^{\infty}(M;SU(2))$ such that
\begin{align}
{\bf g}\cdot(A,\phi)=(A^H,\phi^H)+(\alpha,\rho)
\end{align}
with $(\alpha,\rho)=(\sum_{i=2}^{\infty}a_i(x)y^i,\sum_{i=1}^{\infty}q_i(x)y^i)$ and $a_i,q_i\in \Omega^1 (M;\mathfrak{su}(2))$. We write the left hand side of (3.1) as our new $(A,\phi)$ in the following discussion.\\

By using the parallel transport, we can choose a gauge transformation such that $A$ has no $dy$ part. So one can write $\nabla_A=\bar{\nabla}_A+\frac{\partial}{\partial y}dy$ with $\bar{\nabla}_A$ be the covariant derivative along directions on $S^3$. In addition,
\begin{align}
(A,\phi)=(A^H,\phi^H)+(\alpha,\rho)
\end{align}
with no $dy$ part and $\alpha=O(y^2), \rho=O(y)$. In the rest of our paper, we fix our gauge such that $A_y=0$ and simply denote $\frac{\partial}{\partial y}$ by $\partial_y$.\\

In addition, the maximum principle argument given by Taubes [8] tells us
\begin{pro}
Let $M=Y\times [0,1]$ with a product metric and $(A,\phi=\mathfrak{m}+\phi_ydy)$ be a solution of KW-equation defined on $M$. Then we have the following equation
\begin{align}
\frac{1}{2}(-\partial_y^2+d^*d)|\phi_y|^2+|\partial_y\phi_y|^2+|\nabla_A\phi_y|^2+2|[\phi_y, \mathfrak{m}]|^2=0.
\end{align}
\end{pro}
Therefore, if $|\phi_y|(p)\rightarrow 0$ as $p$ goes to the boundary of $M$, then by standard maximum principle argument, $\phi_y=0$.\\

By using the same argument, S. He proved that $\phi_y=0$ for all $(A,\phi)\in \mathfrak{M}_{K=\emptyset}$ in [4]. This is because $\phi_y=0$ on $y=0$ by Proposition 2.3 and $\lim_{y\rightarrow \infty}\phi_y= 0$ as $(A,\phi)$ converges to the trivial connection.

\begin{remark}We should notice that, for any gauge transformation ${\bf g}\in C^{\infty}(M;SU(2))$,
\begin{align*}
{\bf g}\cdot(A,\phi)=({\bf g}A{\bf g}^{-1}+{\bf g}^{-1}d{\bf g}, {\bf g}\phi{\bf g}^{-1}).
\end{align*}
So the property $\phi_y=0$ is not depending on the gauge we choose.
\end{remark}

\subsection{Weitzenb\"{o}ck formula} Since $(A,\phi)$ is a solution of KW-equation, we have
\begin{align}
&0=\int_{S^3\times\{y>\varepsilon\}}|F_A-\phi^2-*d_A\phi|^2+|d_A*\phi|^2\\
=&\int_{S^3\times\{y>\varepsilon\}} |F_A-\phi^2|^2+2tr(F_A\wedge d_A\phi)-2tr(\phi^2\wedge d_A\phi)+|d_A\phi|^2+|d_A^*\phi|^2\nonumber
\end{align}
(Recall that $|u|^2=-tr(u\wedge *u)$). Applying integration by parts, (3.4) implies
\begin{align}
\int_{S^3\times \{y>\varepsilon\}}|F_A-\phi^2|^2+|d_A\phi|^2+|d_A^*\phi|^2=\frac{2}{3}\int_{S^3}\phi^3\Big\vert_{y=\varepsilon}-2\int_{S^3}\phi\wedge F_A\Big\vert_{y=\varepsilon}.
\end{align}
There will be no boundary term from $y=\infty$ in this equality by Theorem 7.1 in [7]. By Weizenb\"{o}ck formula tells us that
\begin{align}
\int_{S^3\times \{y>\varepsilon\}}|d_A\phi|^2+|d_A^*\phi|^2=\int_{S^3\times \{y>\varepsilon\}}|\nabla_A\phi|^2+Ric(\phi,\phi)+2tr(F_A\wedge *(\phi)^2).
\end{align}
Combine (3.5), (3.6) (and $Ric(\phi,\phi)=2|\phi|^2$), we obtain the following equality
\begin{align}
\int_{S^3\times \{y>\varepsilon\}}|F_A|^2+|\phi^2|^2+|\nabla_A\phi|^2+2|\phi|^2
=\frac{2}{3}\int_{S^3}\phi^3\Big\vert_{y=\varepsilon}-2\int_{S^3}\phi\wedge F_A\Big\vert_{y=\varepsilon}.
\end{align}
We use $*_3$ to denote the Hodge star on $S^3$ with respect to the standard metric. According to Proposition 3.1 and fundamental theorem of calculus, one can derive
\begin{align}
\int_{S^3\times \{y>\varepsilon\}}|\nabla_A\phi|^2+|\phi^2|^2-\frac{2}{3}\int_{S^3}\phi^3\Big\vert_{y=\varepsilon}=\int_{S^3\times \{y>\varepsilon\}}&|\bar{\nabla}_A\phi|^2+|\partial_y\phi|^2+|\phi^2|^2\nonumber\\
-2&\int_{S^3\times \{y>\varepsilon\}}\partial_y\phi\wedge \phi^2 dy\nonumber\\
=\int_{S^3\times \{y>\varepsilon\}}&|\bar{\nabla}_A\phi|^2+|*_3\partial_y \phi+\phi^2|^2.
\end{align}
Therefore we have the following formula
\begin{align}
\int_{S^3\times \{y>\varepsilon\}}|F_A|^2+|\bar{\nabla}_A\phi|^2+|*_3\partial_y \phi+\phi^2|^2+2|\phi|^2=-2\int_{S^3}\phi\wedge F_A\Big\vert_{y=\varepsilon}.
\end{align}
In this equality, $2|\phi|^2$ on the left hand side will blow up as $\varepsilon$ goes to infinity. Meanwhile, $-2\int_{S^3}\phi\wedge F_A\Big\vert_{y=\varepsilon}$ on the right hand side also goes to infinity. We have to show that the difference between them is finite. Here we need to use the following two facts: Firstly, by (3.2),
\begin{align}
\int_{S^3}\phi\wedge F_A\Big\vert_{y=\varepsilon}=\int_{S^3}\phi^H\wedge F_{A^H}\Big\vert_{y=\varepsilon}+O(\varepsilon).
\end{align}
Secondly, one can compute directly to show that $|\phi^H|^2=\frac{1}{y^2}|\omega|^2(1+O(y^2))$ and $F_{A^H}\wedge \phi^H=\frac{1}{y}|\omega|^2(1+O(y^2))d\Omega+dt\wedge \partial_yA^H\wedge \Phi^H$ where $d\Omega$ is the volume form for $S^3$. So we obtain the following result.
\begin{align}
\lim_{\varepsilon\rightarrow 0}\Bigg(\int_{S^3\times \{y>\varepsilon\}}|\phi^H|^2+\int_{S^3}\phi^H\wedge F_{A^H}\Big\vert_{y=\varepsilon}\Bigg)= C_0
\end{align}
where $C_0$ is a constant depending only on $(A^H,\phi^H)$.

 By (3.10) and (3.11), one can derive the following equality from (3.9) (by taking $\varepsilon$ goes to zero):
\begin{align}
\int_{M}|F_A|^2+|\bar{\nabla}_A\phi|^2+|*_3\partial_y \phi+\phi^2|^2-4tr(\phi^H\wedge *\rho)+2|\rho|^2=C_0
\end{align}
Note that every term on the left hand side of (3.12) is positive except the integration against $4tr(\phi^H\wedge \rho)$. Therefore, the following proposition implies Theorem 1.1.
\begin{pro}
Let $(A,\phi)=(A^H,\phi^H)+(\alpha,\rho)$ defined in (3.2). Then
\begin{align*}
\int_{M}|2tr(\phi^H\wedge *\rho)|\leq C_1+\int_M|\rho|^2+\frac{1}{2}\int_M |*_3\partial_y \phi+\phi^2|^2
\end{align*}
for some constant $C_1$ depends only on $(A^H,\phi^H)$.
\end{pro}
\begin{proof}
\ \\
{\bf Step 1}.
First of all, let us use the notations $\mathfrak{t}_i$, $\mathfrak{e}_i$ defined in Section 2. Denote by $p$ the projection $p:M\rightarrow S^3$. We define the following subbundles of $p^*(T^*S^3)\otimes \mathfrak{su}(2)$:
\begin{align}
V^1&=\mbox{real line bundle spanned by }\{\omega\}:=\mbox{span}\{\omega\},\\
V^2&=\mbox{span}\{\mathfrak{t}_i\mathfrak{e}_j-\mathfrak{t}_j\mathfrak{e}_i\}_{i\neq j},\\
V^3&=\mbox{span}\{\mathfrak{t}_1\mathfrak{e}_1-\mathfrak{t}_2\mathfrak{e}_2, \mathfrak{t}_1\mathfrak{e}_1-\mathfrak{t}_3\mathfrak{e}_3,\mathfrak{t}_i\mathfrak{e}_j+\mathfrak{t}_j\mathfrak{e}_i\}_{i\neq j}.
\end{align}
It is easy to check that $V^i\perp V^j$ for any $i\neq j$. Then we have
\begin{align}
\tilde{\Omega}^1(M;\mathfrak{su}(2)):&=\Gamma(p^*(T^*S^3)\otimes \mathfrak{su}(2))\\
&=\Gamma(V^1)\oplus\Gamma(V^2)\oplus \Gamma(V^3).\nonumber
\end{align}
For any $u\in \tilde{\Omega}^1(M;\mathfrak{su}(2))$, we use $u^{(i)}$ to denote the component of $u$ in $\Gamma(V^i)$, so $u=u^{(1)}+u^{(2)}+u^{(3)}$. Recall that $\rho\in \tilde{\Omega}^1(M;\mathfrak{su}(2))$ by Proposition 3.1. So we can write $\rho=\rho^{(1)}+\rho^{(2)}+\rho^{(3)}$ . We should also notice that, for any $v, w\in \tilde{\Omega}^1(M;\mathfrak{su}(2))$, $*_3(v\wedge v), *_3[v,w]\in \tilde{\Omega}^1(M;\mathfrak{su}(2))$. The following property was showed in [6].
\begin{align}
&*_3[\omega, v^{(i)}]=\lambda_i v^{(i)}\in V^i\mbox{ for any }v^{(i)}\in \Gamma(V^i),\\
&\mbox{ } \mbox{ where }(\lambda_1,\lambda_2,\lambda_3)=(2,1,-1).\nonumber
\end{align}
By (3.17) and some computation, we have
\begin{lemma}
\begin{align}
&|(*_3(v\wedge v))^{(1)}-*_3v^{(1)}\wedge v^{(1)}|\leq \frac{1}{\sqrt{6}}\Big(|v^{(2)}|^2+|v^{(3)}|^2\Big).
\end{align}
\end{lemma}
We leave the proof of this lemma in Appendix (Readers can also try to prove it by themselves).\\
\ \\
{\bf Step 2}. Since we have $|\phi^H|\leq C_2 e^{-2y}$ for some $C_2>0$ when $y>1$, so by Cauchy-Schwarz inequality (also remember that $-tr(u\wedge v)$ is the inner product $\langle u,v\rangle$),
\begin{align}
\int_{S^3\times\{y>1\}}|2tr(\phi^H\wedge *\rho)|&\leq \int_{S^3\times\{y>1\}}C_2e^{-2y}|\rho^{(1)}|\nonumber\\
&\leq C_2^2+ \frac{1}{2}\int_{S^3\times\{y>1\}}|\rho^{(1)}|^2.
\end{align}

\ \\
{\bf Step 3}. Here we derive some equality and notations for later use. Let us write $\phi^H=h\omega$ and $\rho^{(1)}=\alpha \omega$ for real-valued functions $h,\alpha:M\rightarrow \mathbb{R}$. Notice that $h$ depends only on the variable $y$. Then we have the following equality
\begin{align}
\partial_y \rho^{(1)}+*_3[\phi^H,\rho^{(1)}]+*_3(\rho^{(1)}\wedge\rho^{(1)})&=\partial_y \rho^{(1)}+2h\rho^{(1)}+\alpha \rho^{(1)}\\
&=f^{-1}\partial_y(f\rho^{(1)})\nonumber
\end{align}
where 
\begin{align*}
f(p,y):=e^{-2\int_y^{\infty}h(s)ds+\int_0^y\alpha(p,s)ds}
\end{align*}
for $(p,y)\in S^3\times \mathbb{R}^+=M$. Recall that, for any real-valued function $W:\mathbb{R}^+\rightarrow \mathbb{R}$, we have $\frac{d}{dy}|W|=sign(W)\frac{d}{dy} W$ (in the sense of weak derivatives, readers can see the proof of Lemma 7.6 in [2] for details). Hence $f>0$ implies that $f^{-1}\partial_y(f|\rho^{(1)}|)=sign(\alpha)f^{-1}\partial_y(f\alpha)|\omega|$. So we can conclude the following equality by using (3.20).
\begin{align}
f^{-1}\partial_y(f|\rho^{(1)}|)=\partial_y|\rho^{(1)}|+2h|\rho^{(1)}|+sign(\alpha)|\alpha|^2|\omega|
\end{align}
almost everywhere.\\

By (3.21) and integration by parts , we have
\begin{align}
&\mbox{ }\mbox{ }\mbox{ }\int_{S^3\times(0,1]}|2tr(\phi^H\wedge *\rho)|\nonumber\\
&\leq \int_{S^3\times(0,1]}2h|\rho^{(1)}||\omega|\nonumber\\
&\leq \int_{S^3\times(0,1]}f^{-1}\partial_y(f|\rho^{(1)}|)|\omega|+\int_{S^3\times(0,1]}|\rho^{(1)}|^2-\int_{S^3\times\{1\}}|\rho^{(1)}||\omega|\nonumber\\
&\leq \int_{S^3\times(0,1]} f^{-1}\partial_y(f|\rho^{(1)}|)|\omega|+\int_{S^3\times(0,1]}|\rho^{(1)}|^2\\
&\leq \int_{S^3\times(0,1]}|*_3\partial_y\rho^{(1)}+[\phi^H,\rho^1]+\rho^{(1)}\wedge\rho^{(1)}||\omega+\int_{S^3\times(0,1]}|\rho^{(1)}|^2\nonumber
\end{align}
where $|\omega|=\sqrt{\frac{3}{2}}$. By (3.18) and triangle inequality again, the integrand of the second last term in (3.22) can be bounded as the following:
\begin{align}
&|*_3\partial_y\rho^{(1)}+[\phi^H,\rho^{(1)}]+\rho^{(1)}\wedge\rho^{(1)}||\omega|\\
\leq& |*_3\partial_y\rho^{(1)}+[\phi^H,\rho^{(1)}]+(\rho\wedge\rho)^{(1)}||\omega|+\frac{1}{2}\Big(|\rho^{(2)}|^2+|\rho^{(3)}|^2\Big).\nonumber
\end{align}

Notice that $(\partial_y\phi+*_3\phi^2)^{(1)}=*_3\phi^H+(\phi^H)^2+*_3\partial_y\rho^{(1)}+[\phi^H,\rho^{(1)}]+(\rho\wedge\rho)^{(1)}$. So one can easily obtain from Proposition 2.5 that 
\begin{align*}
\|*_3\partial_y\rho^{(1)}+[\phi^H,\rho^{(1)}]+(\rho\wedge\rho)^{(1)}\|_{L^1(S^3\times(0,1])}\leq \sqrt{4\pi}C_H+\|*_3\partial_y\phi+\phi^2\|_{L^1(S^3\times(0,1])}.
\end{align*}
 So by Cauchy-Schwarz inequality, (3.23) implies
\begin{align}
&\int_{S^3\times(0,1]}|*_3\partial_y\rho^{(1)}+[\phi^H,\rho^{(1)}]+\rho^{(1)}\wedge\rho^{(1)}||\omega|\\
\leq&\sqrt{4\pi}C_H+ \int_{S^3\times(0,1]}|*_3\partial_y\phi+\phi^2||\omega|+\frac{1}{2}\int_{S^3\times(0,1]}|\rho^{(2)}|^2+|\rho^{(3)}|^2\nonumber\\
\leq&\sqrt{4\pi}C_H+3\pi+ \frac{1}{2}\int_{S^3\times(0,1]}|*_3\partial_y\phi+\phi^2|^2+\frac{1}{2}\int_{S^3\times(0,1]}|\rho^{(2)}|^2+|\rho^{(3)}|^2.\nonumber
\end{align}
So by (3.19), (3.22) and (3.24), we prove the Proposition 3.3. Therefore Theorem 1.1 is proved now. The constant $C_1$ in Proposition 3.3 will be determined by $C_2$ and $C_H$. The constant $C$ in Theorem 1.1 will be determined by $C_0$ and $C_1$.
\end{proof}
\begin{remark}
In addition, we can actually prove
\begin{align}
\int_{M}|F_A|^2+|\bar{\nabla}_A\phi|^2+\frac{1}{2}|*_3\partial_y \phi+\phi^2|^2\leq C
\end{align}
by putting different coefficient weight in (3.24) and changing the coefficient of \\
$\int_M|*_3\partial_y \phi+\phi^2|^2$ in Proposition 3.3 by $\frac{1}{4}$.
\end{remark}
\begin{remark}
In the case that $K\neq \emptyset$, one can observe from the computation of the modal solution $(A^L,\phi^L)$ that $F_A$ is not $L^2$ bounded. So the argument we use in this paper cannot apply to the general cases. However, $|F^+_A|$ and $|F^-_A|$ blow up at the same rate as $y\rightarrow 0$. Therefore, one should seek for a formula which directly proves that $|p(A)|$ has a uniform bound for all $(A,\phi)\in \mathfrak{M}_K$.\\
\end{remark}

\section{Appendix: Proof of Lemma 3.4}
To prove (3.18), we write $v=v^{(1)}+v^{(2)}+v^{(3)}$ where $v^{(1)}=\alpha_1\omega$ for some $\alpha_1\in \mathbb{R}$. By (3.17), we will have
\begin{align*}
(*_3(v\wedge v))^{(1)}=(*_3(v^{(1)}\wedge v^{(1)}&+v^{(2)}\wedge v^{(2)}+v^{(3)}\wedge v^{(3)}\\
&+[v^{(1)},v^{(2)}]+[v^{(1)},v^{(3)}]+[v^{(2)},v^{(3)}]))^{(1)}\\
=(*_3(v^{(1)}\wedge v^{(1)}&+v^{(2)}\wedge v^{(2)}+v^{(3)}\wedge v^{(3)}+2[v^{(2)},v^{(3)}]))^{(1)}
\end{align*}
One can easily check by the definitions (3.14) and (3.15), $[v^{(2)},v^{(3)}]\perp V^1$. So
\begin{align*}
(*_3(v\wedge v))^{(1)}-*_3v^{(1)}\wedge v^{(1)}=(v^{(2)}\wedge v^{(2)}+v^{(3)}\wedge v^{(3)})^{(1)}.
\end{align*}
By this equality, to obtain (3.18), we have to prove that for any $v^{(2)}\in \Gamma(V^2)$, $v^{(3)}\in \Gamma(V^3)$,
\begin{align}
|(v^{(2)}\wedge v^{(2)})^{(1)}|= \frac{1}{\sqrt{6}}|v^{(2)}|^2;\\
|(v^{(3)}\wedge v^{(3)})^{(1)}|= \frac{1}{\sqrt{6}}|v^{(3)}|^2.
\end{align}
Let $\mu_1=\mathfrak{t}_2\mathfrak{e}_3-\mathfrak{t}_3\mathfrak{e}_2$, $\mu_2=\mathfrak{t}_3\mathfrak{e}_1-\mathfrak{t}_1\mathfrak{e}_3$ and $\mu_3=\mathfrak{t}_1\mathfrak{e}_2-\mathfrak{t}_2\mathfrak{e}_1$. Then we have
\begin{align*}
*_3(\mu_1\wedge \mu_1)=\mathfrak{t}_1\mathfrak{e}_1,
*_3(\mu_2\wedge \mu_2)&=\mathfrak{t}_2\mathfrak{e}_2,
*_3(\mu_3\wedge \mu_3)=\mathfrak{t}_3\mathfrak{e}_3,\\
*_3[\mu_i,\mu_j]&=\mathfrak{t}_i\mathfrak{e}_j+\mathfrak{t}_i\mathfrak{e}_j, i\neq j.
\end{align*}
Therefore $*_3[\mu_i,\mu_j]\perp V^1$ for all $i\neq j$. Also notice that $\mathfrak{t}_1\mathfrak{e}_1=\frac{1}{3}(\omega+ (\mathfrak{t}_1\mathfrak{e}_1-\mathfrak{t}_2\mathfrak{e}_2)+(\mathfrak{t}_1\mathfrak{e}_1-\mathfrak{t}_3\mathfrak{e}_3))$. So $(*_3(\mu_1\wedge \mu_1))^{(1)}=\frac{1}{3}\omega$. Similarly, $(*_3(\mu_2\wedge \mu_2))^{(1)}=\frac{1}{3}\omega$ and $(*_3(\mu_3\wedge \mu_3))^{(1)}=\frac{1}{3}\omega$.\\

Combine these results and $|\mu_i|=1$, $|\omega|=\sqrt{\frac{3}{2}}$, for any $v^{(2)}=\alpha_1\mu_1+\alpha_2\mu_2+\alpha_3\mu_3$, we have
\begin{align*}
|(v^{(2)}\wedge v^{(2)})^{(1)}|=|\frac{1}{3}\omega|\sum_{i=1}^3 |\alpha_i|^2=\frac{1}{\sqrt{6}}\sum_{i=1}^3 |\alpha_i|^2=\frac{1}{\sqrt{6}}|v^{(2)}|^2.
\end{align*}
So we prove (4.1).\\

Let $\nu_1=\mathfrak{t}_2\mathfrak{e}_3+\mathfrak{t}_3\mathfrak{e}_2$, $\nu_2=\mathfrak{t}_3\mathfrak{e}_1+\mathfrak{t}_1\mathfrak{e}_3$, $\nu_3=\mathfrak{t}_1\mathfrak{e}_2+\mathfrak{t}_2\mathfrak{e}_1$, $\nu_{1,2}=\mathfrak{t}_1\mathfrak{e}_1-\mathfrak{t}_2\mathfrak{e}_2$ and $\nu_{1,3}=\mathfrak{t}_1\mathfrak{e}_1-\mathfrak{t}_3\mathfrak{e}_3$. Then we will have
\begin{align*}
*_3(\nu_1\wedge \nu_1)=\mathfrak{t}_1\mathfrak{e}_1,
*_3(\nu_2\wedge \nu_2)&=\mathfrak{t}_2\mathfrak{e}_2,
*_3(\nu_3\wedge \nu_3)=\mathfrak{t}_3\mathfrak{e}_3,\\
*_3(\nu_{1,2}\wedge \nu_{1,2})=\mathfrak{t}_3\mathfrak{e}_3,&\mbox{ }*_3(\nu_{1,3}\wedge \nu_{1,3})=\mathfrak{t}_2\mathfrak{e}_2,\\
*_3[\mu_a,\mu_b]&\perp V^1, a\neq b.
\end{align*}
So for any $v^{(3)}=\beta_1\nu_1+\beta_2\nu_2+\beta_3\nu_3+\beta_{1,2}\nu_{1,2}+\beta_{1,3}\nu_{1,3}$, we have
\begin{align*}
|(v^{(3)}\wedge v^{(3)})^{(1)}|&=|\frac{1}{3}\omega|(\sum_{i=1}^3 |\beta_i|^2+|\beta_{1,2}|^2+|\beta_{1,3}|^2)\\
&=\frac{1}{\sqrt{6}}(\sum_{i=1}^3 |\beta_i|^2+|\beta_{1,2}|^2+|\beta_{1,3}|^2)=\frac{1}{\sqrt{6}}|v^{(3)}|^2.
\end{align*}
Therefore, we obtain (4.2).\\

\noindent {\bf Acknowledgement:} The second author would like to thank Siqi He for several comments and explanations of his works. The work of N. Leung described in this paper was substantially supported by grants from the Research Grants Council of the Hong Kong Special Administrative Region, China (Project No. CUHK14303516) and a CUHK direct grant (Project No. 4053215). \\ 

\bibliographystyle{amsplain}

\end{document}